\documentclass{amsart}
\usepackage{setspace}
\usepackage{a4}
\usepackage{amsthm}
\usepackage{latexsym}
\usepackage{amsfonts}
\usepackage{graphicx}
\usepackage{textcomp}
\usepackage{cite}
\usepackage{enumerate}
\usepackage{amssymb}
\usepackage{hyperref}
\usepackage{amsmath}
\usepackage{tikz}
\usepackage[mathscr]{euscript}
\usepackage{mathtools}
\newtheorem{theorem}{Theorem}[section]

\newtheorem{corollary}[theorem] {Corollary}

\newtheorem{problem}[theorem]{Problem}

\title{This is the title}
\usepackage{fancyhdr}

\begin{document}
	\vspace{0.9cm}	
\hrule\hrule\hrule\hrule\hrule
\vspace{0.3cm}	
\begin{center}
{\bf{Quaternionic Kittaneh Numerical Radius Inequalities}}\\
\vspace{0.3cm}
\hrule\hrule\hrule
\vspace{0.3cm}
\textbf{K. Mahesh Krishna}\\
School of Mathematics and Natural Sciences\\
Chanakya University Global Campus\\
NH-648, Haraluru Village\\
Devanahalli Taluk, 	Bengaluru  North District\\
Karnataka State 562 110 India \\
Email: kmaheshak@gmail.com\\

Date: \today
\hrule\hrule
\end{center}

\vspace{0.5cm}
\textbf{Abstract}: We show that for certain classes of matrices over quaternions, a more restrictive version of breakthrough numerical radius inequalities obtained by Kittaneh [\textit{{Stud. Math.}, 2005}] hold.



\hrule

\hrule
\section{Introduction}
Let $\mathbb{C}^n$ be the standard Hilbert space (with the inner product linear in first variable). Recall that 
the \textbf{numerical range} or the \textbf{field of values}  \cite{GUSTAFSONRAO, HALMOS, HORNJOHNSON, GOLDBERGTADMOR, GAUWUBOOK, TOEPLITZ, HAUSDORFF, ISTRATESCU, HALMOS2} of a matrix $T \in M_n(\mathbb{C})$ is defined as 
\begin{align*}
	W_\mathbb{C}(T)\coloneqq \{\langle Th, h\rangle: h \in \mathbb{C}^n, \|h\|=1\} \subseteq \mathbb{C}
\end{align*}
and the \textbf{numerical radius} of  $T$ is defined as 
\begin{align*}
	w_\mathbb{C}(T)\coloneqq \sup_{\lambda \in W_\mathbb{C}(T)}|\lambda|=\sup_{h \in \mathbb{C}^n, \|h\|=1} |\langle Th, h \rangle |.
\end{align*}
The celebrated Toeplitz-Hausdorff theorem says that the numerical range is always a convex subset of  $\mathbb{C}$ \cite{HAUSDORFF, TOEPLITZ, HALMOS, HORNJOHNSON, GUSTAFSONRAO, GAUWUBOOK, LI}.  Elementary calculations using the Cauchy-Schwarz inequality and the generalized polarization identity  \cite{GUSTAFSONRAO} show that
\begin{align}\label{FI}
	\frac{\|T\|}{2}\leq w_\mathbb{C}(T)\leq \|T\|.
\end{align}
 Inequalities (\ref{FI}) say that the map $w_\mathbb{C}(\cdot):M_n(\mathbb{C})\to \mathbb{R}$ is a norm which is equivalent to the operator norm.  
In 2005, Kittaneh made a far reaching generalization of Inequalities (\ref{FI}) \cite{KITTANEH}.
\begin{theorem} \cite{KITTANEH} (\textbf{Kittaneh Numerical Radius Inequalities}) \label{KT}
	Let $T \in M_n(\mathbb{C})$. Then 
\begin{align}\label{KI}
	\frac{1}{2}\sqrt{\|T^*T+TT^*\|}\leq w_\mathbb{C}(T)\leq 	\frac{1}{\sqrt{2}}\sqrt{\|T^*T+TT^*\|}.	
\end{align}
\end{theorem}
Since 
\begin{align*}
			\frac{\|T\|}{2}\leq \frac{1}{2}\sqrt{\|T^*T+TT^*\|}\leq	\frac{1}{\sqrt{2}}\sqrt{\|T^*T+TT^*\|}\leq \|T\|, \quad \forall T \in M_n(\mathbb{C}),
\end{align*}
Inequalities (\ref{KI}) improve Inequalities (\ref{FI}).
Even though quaternionic numerical ranges are well-studied in the literature  \cite{YEUNG, ZHANG, SOTHOMPSONZHANG, SO, KIPPENHAHN, SOTHOMSON, THOMPSON, CARVALHODIOGOMENDES, CARVALHODIOGOMENDES2, SO2, CARVALHODIOGOMENDES3, YEUNGSIU}, to the author's knowledge, quaternionic version of Inequalities (\ref{KI}) are not known. 

In this  note, we show that for certain classes of quaternionic matrices, a more restrictive version of Inequalities  (\ref{KI}) hold. We recall the set up. 
Let  
\begin{align*}
	\mathbb{H}\coloneqq \{a+ib+jc+kd: a, b, c, d \in \mathbb{R}, i^2=j^2=k^2=ijk=-1\}
\end{align*}
be the skew-field of quaternion numbers with the conjugation
\begin{align*}
	\overline{a+ib+jc+kd}\coloneqq a-ib-jc-kd, \quad \forall a+ib+jc+kd \in 	\mathbb{H}
\end{align*}
and the modulus 
\begin{align*}
	|a+ib+jc+kd|\coloneqq \sqrt{a^2+b^2+c^2+d^2}, \quad \forall a+ib+jc+kd \in 	\mathbb{H}.	
\end{align*}
Let $\mathbb{H}^n$ be the standard right quaternionic Hilbert space equipped with the inner product 
\begin{align*}
	\langle (z_r)_{r=1}^n, (w_r)_{r=1}^n\rangle \coloneqq \sum_{r=1}^{n}\overline{z_r}w_r, \quad \forall (z_r)_{r=1}^n, (w_r)_{r=1}^n\in \mathbb{H}^n
\end{align*}
and the norm 
\begin{align*}
	\|(z_r)_{r=1}^n\|\coloneqq \left(\sum_{r=1}^{n}|z_r|^2\right)^\frac{1}{2}, \quad \forall (z_r)_{r=1}^n\in \mathbb{H}^n.
\end{align*}
Given a matrix $T \in M_n(\mathbb{H})$, its operator norm is defined as 
\begin{align*}
	\|T\|\coloneqq \sup_{h \in \mathbb{H}^n, \|h\|=1}\|Th\|.
\end{align*}
We derive quaternionic version of Theorem \ref{KT} in Theorem \ref{QKT}.

\section{Quaternionic Kittaneh Numerical Radius Inequalities}

Similar to the Hilbert space case, the \textbf{quaternionic numerical range} \cite{YEUNG, ZHANG, SOTHOMPSONZHANG, SO, KIPPENHAHN, SOTHOMSON, THOMPSON, CARVALHODIOGOMENDES, CARVALHODIOGOMENDES2, SO2, CARVALHODIOGOMENDES3, YEUNGSIU}
of a quaternionic matrix $T \in M_n(\mathbb{H})$ is defined as 
\begin{align*}
	W_\mathbb{H}(T)\coloneqq \{\langle h,Th\rangle: h \in \mathbb{H}^n, \|h\|=1\}\subseteq \mathbb{H}
\end{align*}
and the \textbf{quaternionic numerical radius} of $T$ is defined as 
\begin{align*}
	w_\mathbb{H}(T)\coloneqq \sup_{\lambda \in W_\mathbb{H}(T)}|\lambda|=\sup_{h \in \mathbb{H}^n, \|h\|=1}|\langle h, Th\rangle|.
\end{align*}
Unlike in the complex case, the quaternionic numerical range need not be convex (meaning the Toeplitz-Hausdorff theorem fails) \cite{HORNJOHNSON}. 
Let $T \in M_n(\mathbb{H})$ be self-adjoint.  We note that
\begin{align}\label{ST}
	\|T\|=\sup_{h \in \mathbb{H}^n, \|h\|=1}|\langle h, Th\rangle|.
\end{align}
In fact, define $M\coloneqq \sup_{h \in \mathbb{H}^n, \|h\|=1}|\langle h, Th\rangle|$. The we have
\begin{align*}
	M= \sup_{h \in \mathbb{H}^n, \|h\|=1}|\langle h, Th\rangle|\leq \sup_{h \in \mathbb{H}^n, \|h\|=1}\|h\| \|Th\|=\sup_{h \in \mathbb{H}^n, \|h\|=1}\|Th\|=\|T\|.
\end{align*}
Now let $x, y \in \mathbb{H}^n$. Then 
\begin{align*}
	\langle x+y, T(x+y)\rangle -\langle x-y, T(x-y)\rangle =2(\langle x, Ty \rangle +\langle y, Tx \rangle).
\end{align*}
Therefore
\begin{align*}
	|(\langle x, Ty \rangle +\langle y, Tx \rangle|&\leq \frac{1}{2}(|\langle x+y, T(x+y)\rangle|+|\langle x-y, T(x-y)\rangle|)\\
	&\leq  \frac{1}{2}(M\|x+y\|^2+M\|x-y\|^2)\\
	&=M(\|x\|^2+\|y\|^2).
\end{align*}
Since, $T$ is self-adjoint, 
\begin{align}\label{AE}
		|(\langle x, Ty \rangle +\langle Ty, x \rangle|\leq M(\|x\|^2+\|y\|^2), \quad \forall x, y \in \mathbb{H}^n.
\end{align}
For a given $h \in \mathbb{H}^n$ with $ \|h\|=1$ and $\|Th\|\neq 0$,  we take 
\begin{align*}
	x=\frac{Th}{\|Th\|}, \quad y = h
\end{align*}
in Inequality (\ref{AE}). Then we have 
\begin{align*}
	2\|Th\|\leq 2M.
\end{align*}
By varying $h$ we get $\|T\|\leq M$. Therefore $\|T\|=M$. \\
We now derive our main theorem using Equation (\ref{ST}).
 In the paper, the anti-commutator of $T, S \in M_n(\mathbb{H})$ is defined as $\{T, S\}\coloneqq TS+ST$.
\begin{theorem}\label{QKT}
(\textbf{Quaternionic Kittaneh Numerical Radius Inequalities})
Let $T \in M_n(\mathbb{H})$. Assume that 
\begin{align*}
	T=A+iB+jC+kD, \text{ for some } A, B, C ,D \in M_n(\mathbb{H})
\end{align*}
satisfying the following conditions. 
	\begin{enumerate}[\upshape(i)]
		\item $A^*=A$, $B^*=B$, $C^*=C$, $D^*=D$.
	\item $iB=Bi$, $jC=Cj$, $Dk=kD$.
	\item $\{iB, jC\}+\{iB, kD\}+\{jC, kD\}=0$.
		\end{enumerate}
Then 
\begin{align}\label{MI}
	\frac{1}{2\sqrt{2}}\sqrt{\|T^*T+TT^*\|}\leq w_\mathbb{H}(T)\leq\frac{1}{\sqrt{2}}\sqrt{\|T^*T+TT^*\|}.	
\end{align}	
\end{theorem}
\begin{proof}
	We have 
	\begin{align*}
		T=A+iB+jC+kD, \quad T^*=A-iB-jC-kD.
	\end{align*}
Therefore 
\begin{align*}
	T^*T+TT^*&=(A-iB-jC-kD)(A+iB+jC+kD)+(A+iB+jC+kD)(A-iB-jC-kD)\\
	&=2(A^2+B^2+C^2+D^2)-2(\{iB, jC\}+\{iB, kD\}+\{jC, kD\})\\
	&=2(A^2+B^2+C^2+D^2).
\end{align*}
For $h \in \mathbb{H}^n$ with $\|h\|=1$, we have
\begin{align*}
|\langle h, Th \rangle |^2&=|\langle h, (A+iB+jC+kD)h \rangle |^2\\
&=|\langle h, Ah \rangle +i\langle h, Bh \rangle +j\langle h, Ch \rangle +k\langle h, Dh \rangle |^2\\
&=|\langle h, Ah \rangle |^2+|\langle h, Bh \rangle |^2+|\langle h, Ch \rangle |^2+|\langle h, Dh \rangle |^2\\
&\leq \|Ah\|^2+\|Bh\|^2+\|Ch\|^2+\|Dh\|^2\\
&=\langle h, A^2h \rangle +\langle h, B^2h \rangle +\langle h, C^2h \rangle +\langle h, D^2h \rangle\\
&=\langle h, (A^2+B^2+C^2+D^2)h \rangle.
\end{align*}
Since $A, B, C, D $ are self-adjoint, we have
\begin{align*}
	w_\mathbb{H}(T)^2&= \sup_{h \in \mathbb{H}^n, \|h\|=1} |\langle h, Th \rangle |^2\\
	       &\leq \sup_{h \in \mathbb{H}^n, \|h\|=1} \langle h, (A^2+B^2+C^2+D^2)h \rangle\\
	       &=\|A^2+B^2+C^2+D^2\|\\
	       &=\frac{1}{2}\|	T^*T+TT^*\|.
\end{align*}
On the other hand, for $h \in \mathbb{H}^n$ with $\|h\|=1$, using the Cauchy-Schwarz inequality, we also have 
\begin{align*}
|\langle h, Th \rangle |^2&=|\langle h, (A+iB+jC+kD)h \rangle |^2\\
&=|\langle h, Ah \rangle +i\langle h, Bh \rangle +j\langle h, Ch \rangle +k\langle h, Dh \rangle |^2\\
&=|\langle h, Ah \rangle |^2+|\langle h, Bh \rangle |^2+|\langle h, Ch \rangle |^2+|\langle h, Dh \rangle |^2\\
&\geq \frac{1}{4}(|\langle h, Ah \rangle |+|\langle h, Bh \rangle |+|\langle h, Ch \rangle |+|\langle h, Dh \rangle|)^2\\
&\geq\frac{1}{4}|\langle h, (\pm A\pm B\pm C\pm D)h \rangle |^2.
\end{align*}
Since $A, B, C, D $ are self-adjoint, we have
\begin{align*}
	w_\mathbb{H}(T)^2&= \sup_{h \in \mathbb{H}^n, \|h\|=1} |\langle h, Th \rangle |^2\\
	&\geq \frac{1}{4}\sup_{h \in \mathbb{H}^n, \|h\|=1}|\langle h, (\pm A\pm B\pm C\pm D)h \rangle |^2\\
	&= \frac{1}{4}\|\pm A\pm B\pm C\pm D\|^2\\
	&=\frac{1}{4}\|(\pm A\pm B\pm C\pm D)^2\|.
\end{align*}
Therefore 
\begin{align*}
	&w_\mathbb{H}(T)^2\geq 	\\
	& \frac{1}{16}(\|(-A+B+C+D)^2\|+\|(A-B+C+D)^2\|+\|(A+B-C+D)^2\|+\|(A+B+C-D)^2\|)\\
	&\geq \frac{1}{16}\|(-A+B+C+D)^2+(A-B+C+D)^2+(A+B-C+D)^2+(A+B+C-D)^2\|\\
	&=\frac{1}{4}\|A^2+B^2+C^2+D^2\|\\
	&=\frac{1}{8}\|T^*T+TT^*\|.
\end{align*}
\end{proof}
\begin{corollary}
Let $T \in M_n(\mathbb{H})$. Assume that 
\begin{align*}
	T=A+iB+jC+kD, \text{ for some } A, B, C ,D \in M_n(\mathbb{H})
\end{align*}
satisfying the following conditions. 
\begin{enumerate}[\upshape(i)]
	\item $A^*=A$, $B^*=B$, $C^*=C$, $D^*=D$.
	\item $iB=Bi$, $jC=Cj$, $Dk=kD$.
	\item $\{iB, jC\}=\{iB, kD\}=\{jC, kD\}=0$.
\end{enumerate}
Then 
\begin{align*}
	\frac{1}{2\sqrt{2}}\sqrt{\|T^*T+TT^*\|}\leq w_\mathbb{H}(T)\leq\frac{1}{\sqrt{2}}\sqrt{\|T^*T+TT^*\|}.	
\end{align*}		
\end{corollary}
Following corollary says that using self adjoint real matrices, we can easily construct quaternionic matrices which satisfy  Inequalities (\ref{MI}). 
\begin{corollary}
Let $A, B, C ,D \in M_n(\mathbb{R})$ satisfying the  following conditions.  
\begin{enumerate}[\upshape(i)]
	\item $A^t=A$, $B^t=B$, $C^t=C$, $D^t=D$.
	\item $\{B, C\}+\{B, D\}+\{C, D\}=0$.
\end{enumerate}
Define 
\begin{align*}
	T\coloneqq A+iB+jC+kD \in M_n(\mathbb{H}).
\end{align*}
Then 
\begin{align*}
	\frac{1}{2\sqrt{2}}\sqrt{\|T^*T+TT^*\|}\leq w_\mathbb{H}(T)\leq\frac{1}{\sqrt{2}}\sqrt{\|T^*T+TT^*\|}.	
\end{align*}		
\end{corollary}
Kittaneh actually derived a much more general result showing that Theorem \ref{KT} holds for any bounded linear operator on a complex Hilbert space. We state here the quaternionic Hilbert space version of Theorem \ref{QKT}, whose proof is similar to that of Theorem \ref{QKT}. For the definition and basics of quaternionic Hilbert spaces, we refer the reader to \cite{GHILONIMORETTIPEROTTI}.
\begin{theorem}
	Let $\mathcal{Q}$ be a quaternionic Hilbert space. Let $T:\mathcal{Q} \to \mathcal{Q} $ be a bounded linear operator. Assume that 
	\begin{align*}
		T=A+iB+jC+kD, \text{ for some bounded linear operators } A, B, C ,D: \mathcal{Q} \to \mathcal{Q}
	\end{align*}
	satisfying the following conditions. 
	\begin{enumerate}[\upshape(i)]
		\item $A^*=A$, $B^*=B$, $C^*=C$, $D^*=D$.
		\item $iB=Bi$, $jC=Cj$, $Dk=kD$.
		\item $\{iB, jC\}+\{iB, kD\}+\{jC, kD\}=0$.
	\end{enumerate}
	Then 
	\begin{align*}
		\frac{1}{2\sqrt{2}}\sqrt{\|T^*T+TT^*\|}\leq w_\mathbb{H}(T)\leq\frac{1}{\sqrt{2}}\sqrt{\|T^*T+TT^*\|}.	
	\end{align*}
\end{theorem}
\begin{corollary}
	Let $\mathcal{H}$ be a real Hilbert space. Let  $A, B, C, D:\mathcal{H} \to \mathcal{H} $ be  bounded linear operators satisfying the following conditions. 
	\begin{enumerate}[\upshape(i)]
		\item $A^*=A$, $B^*=B$, $C^*=C$, $D^*=D$.
		\item $iB=Bi$, $jC=Cj$, $Dk=kD$.
		\item $\{iB, jC\}+\{iB, kD\}+\{jC, kD\}=0$.
	\end{enumerate}
Define 
\begin{align*}
	T\coloneqq A+iB+jC+kD.
\end{align*}
	Then 
	\begin{align*}
		\frac{1}{2\sqrt{2}}\sqrt{\|T^*T+TT^*\|}\leq w_\mathbb{H}(T)\leq\frac{1}{\sqrt{2}}\sqrt{\|T^*T+TT^*\|}.	
	\end{align*}	
\end{corollary}
We wish to note that Abu-Omar and Kittaneh \cite{ABUOMARKITTANEH} and Bhunia, Bag and Paul \cite{BHUNIABAGPAUL} improved Theorem \ref{KT} but the techniques used by them do not seem to carry over to quaternionic Hilbert spaces.

\section{Conclusions}
	\begin{enumerate}[\upshape(i)]
	\item In 2005, Kittaneh derived a breakthrough generalization of classical numerical radius inequalities \cite{KITTANEH}.
	\item In this article, we showed that a constrained version of Kittaneh numerical radius inequalities can be obtained for certain classes of matrices over quaternions.
\end{enumerate}

\section{Quaternionic Crouzeix Problem}
Breakthrough Crouzeix theorem says following. 
\begin{theorem} \label{CT} \cite{CROUZEIX, CROUZEIX2, CROUZEIX3, BICKELGORKIN, RANSFORDSCHWENNINGER} (\textbf{Crouzeix Theorem}) 
	For every $d \in \mathbb{N}$ and for every matrix $M\in \mathbb{M}_{d}(\mathbb{C})$, we have 
	\begin{align*}
		(\text{\textbf{Crouzeix-Palencia Inequality}})	\quad \quad \quad 	\|p(M)\|\leq (1+\sqrt{2}) \sup\left\{|p(z)|: z \in W_\mathbb{C}(M)\right\},  \quad \forall p \in \mathbb{C}[z].
	\end{align*}	
\end{theorem}
Based on Theorem \ref{CT}, we formulate following problem. 
\begin{problem} (\textbf{Quaternionic Crouzeix Problem})
	Whether there is a universal constant $R_\mathbb{H}$  satisfying following: For every $d \in \mathbb{N}$ and for every $A\in \mathbb{M}_{d}(\mathbb{H})$, we have 
	\begin{align*}
		\|p(A)\|\leq R_\mathbb{H} \sup\left\{|p(z)|: z \in W_\mathbb{H}(A)\right\},  \quad \forall p(z)\coloneqq a_0+a_1z+a_2z^2+\cdots +a_nz^n \in \mathbb{H}[z], n \in \mathbb{N}
	\end{align*}
	where $\mathbb{H}[z]$ is the set of all polynomials over $\mathbb{H}$. 
	We conjecture that $R_\mathbb{H}=4$.
\end{problem}

 \bibliographystyle{plain}
 \bibliography{reference.bib}

\begin{thebibliography}{10}

\bibitem{ABUOMARKITTANEH}
Amer Abu-Omar and Fuad Kittaneh.
\newblock Upper and lower bounds for the numerical radius with an application
  to involution operators.
\newblock {\em Rocky Mt. J. Math.}, 45(4):1055--1064, 2015.

\bibitem{YEUNG}
Yik-Hoi Au-Yeung.
\newblock On the convexity of numerical range in quaternionic {Hilbert} spaces.
\newblock {\em Linear Multilinear Algebra}, 16:93--100, 1984.

\bibitem{YEUNGSIU}
Yik-Hoi Au-Yeung and Lok-Shun Siu.
\newblock Quaternionic numerical range and real subspaces.
\newblock {\em Linear Multilinear Algebra}, 45(4):317--327, 1999.

\bibitem{BHUNIABAGPAUL}
Pintu Bhunia, Santanu Bag, and Kallol Paul.
\newblock Numerical radius inequalities and its applications in estimation of
  zeros of polynomials.
\newblock {\em Linear Algebra Appl.}, 573:166--177, 2019.

\bibitem{BICKELGORKIN}
Kelly Bickel, Pamela Gorkin, Anne Greenbaum, Thomas Ransford, Felix~L.
  Schwenninger, and Elias Wegert.
\newblock Crouzeix's conjecture and related problems.
\newblock {\em Comput. Methods Funct. Theory}, 20(3-4):701--728, 2020.

\bibitem{CARVALHODIOGOMENDES2}
Lu{\'{\i}}s Carvalho, Cristina Diogo, and S{\'e}rgio Mendes.
\newblock On the convexity and circularity of the numerical range of nilpotent
  quaternionic matrices.
\newblock {\em New York J. Math.}, 25:1385--1404, 2019.

\bibitem{CARVALHODIOGOMENDES}
Lu{\'{\i}}s Carvalho, Cristina Diogo, and S{\'e}rgio Mendes.
\newblock The star-center of the quaternionic numerical range.
\newblock {\em Linear Algebra Appl.}, 603:166--185, 2020.

\bibitem{CARVALHODIOGOMENDES3}
Lu{\'{\i}}s Carvalho, Cristina Diogo, and S{\'e}rgio Mendes.
\newblock A new perspective on the quaternionic numerical range of normal
  matrices.
\newblock {\em Linear Multilinear Algebra}, 70(20):5068--5074, 2022.

\bibitem{CROUZEIX3}
M.~Crouzeix and C.~Palencia.
\newblock The numerical range is a {{\((1+\sqrt{2})\)}}-spectral set.
\newblock {\em SIAM J. Matrix Anal. Appl.}, 38(2):649--655, 2017.

\bibitem{CROUZEIX}
Michel Crouzeix.
\newblock Bounds for analytical functions of matrices.
\newblock {\em Integral Equations Oper. Theory}, 48(4):461--477, 2004.

\bibitem{CROUZEIX2}
Michel Crouzeix.
\newblock Numerical range and functional calculus in {Hilbert} space.
\newblock {\em J. Funct. Anal.}, 244(2):668--690, 2007.

\bibitem{GAUWUBOOK}
Hwa-Long Gau and Pei~Yuan Wu.
\newblock {\em Numerical ranges of {Hilbert} space operators}, volume 179 of
  {\em Encycl. Math. Appl.}
\newblock Cambridge: Cambridge University Press, 2021.

\bibitem{GHILONIMORETTIPEROTTI}
Riccardo Ghiloni, Valter Moretti, and Alessandro Perotti.
\newblock Continuous slice functional calculus in quaternionic {Hilbert}
  spaces.
\newblock {\em Rev. Math. Phys.}, 25(4):83, 2013.
\newblock Id/No 1350006.

\bibitem{GOLDBERGTADMOR}
Moshe Goldberg and Eitan Tadmor.
\newblock On the numerical radius and its applications.
\newblock {\em Linear Algebra Appl.}, 42:263--284, 1982.

\bibitem{GUSTAFSONRAO}
Karl~E. Gustafson and Duggirala K.~M. Rao.
\newblock {\em Numerical range. {The} field of values of linear operators and
  matrices}.
\newblock Universitext. New York, NY: Springer, 1996.

\bibitem{HALMOS2}
P.~R. Halmos.
\newblock Numerical ranges and normal dilations.
\newblock {\em Acta Sci. Math.}, 25:1--5, 1964.

\bibitem{HALMOS}
Paul~R. Halmos.
\newblock {\em A {Hilbert} space problem book}, volume~19 of {\em Grad. Texts
  Math.}
\newblock Springer, Cham, 1982.

\bibitem{HAUSDORFF}
F.~Hausdorff.
\newblock Der {Wertvorrat} einer {Bilinearform}.
\newblock {\em Math. Z.}, 3:314--316, 1919.

\bibitem{HORNJOHNSON}
Roger~A. Horn and Charles~R. Johnson.
\newblock {\em Topics in matrix analysis}.
\newblock Cambridge University Press, 1991.

\bibitem{ISTRATESCU}
Vasile~I. Istratescu.
\newblock {\em Introduction to linear operator theory}, volume~65 of {\em Pure
  Appl. Math., Marcel Dekker}.
\newblock Marcel Dekker, Inc., New York, NY, 1981.

\bibitem{KIPPENHAHN}
Rudolf Kippenhahn.
\newblock On the numerical range of a matrix.
\newblock {\em Linear Multilinear Algebra}, 56(1-2):185--225, 2008.

\bibitem{KITTANEH}
Fuad Kittaneh.
\newblock Numerical radius inequalities for {Hilbert} space operators.
\newblock {\em Stud. Math.}, 168(1):73--80, 2005.

\bibitem{LI}
Chi-Kwong Li.
\newblock A simple proof of the elliptical range theorem.
\newblock {\em Proc. Am. Math. Soc.}, 124(7):1985--1986, 1996.

\bibitem{RANSFORDSCHWENNINGER}
Thomas Ransford and Felix~L. Schwenninger.
\newblock Remarks on the {Crouzeix}-{Palencia} proof that the numerical range
  is a $(1+\sqrt{2})$-spectral set.
\newblock {\em SIAM J. Matrix Anal. Appl.}, 39(1):342--345, 2018.

\bibitem{SO}
Wasin So.
\newblock An explicit criterion for the convexity of quaternionic numerical
  range.
\newblock {\em Can. Math. Bull.}, 41(1):105--108, 1998.

\bibitem{SO2}
Wasin So.
\newblock Convexity conditions for 2 $\times$ 2 quaternionic numerical range.
\newblock {\em Southeast Asian Bull. Math.}, 47(5):707--715, 2023.

\bibitem{SOTHOMSON}
Wasin So and Robert~C. Thompson.
\newblock Convexity of the upper complex plane part of the numerical range of a
  quaternionic matrix.
\newblock {\em Linear Multilinear Algebra}, 41(4):303--365, 1996.

\bibitem{SOTHOMPSONZHANG}
Wasin So, Robert~C. Thompson, and Fuzhen Zhang.
\newblock The numerical range of normal matrices with quaternion entries.
\newblock {\em Linear Multilinear Algebra}, 37(1-3):175--195, 1994.

\bibitem{THOMPSON}
Robert~C. Thompson.
\newblock The upper numerical range of a quaternionic matrix is not a complex
  numerical range.
\newblock {\em Linear Algebra Appl.}, 254:19--28, 1997.

\bibitem{TOEPLITZ}
O.~Toeplitz.
\newblock Das algebraische {Analogon} zu einem {Satze} von {Fej{\'e}r}.
\newblock {\em Math. Z.}, 2:187--197, 1918.

\bibitem{ZHANG}
Fuzhen Zhang.
\newblock Quaternions and matrices of quaternions.
\newblock {\em Linear Algebra Appl.}, 251:21--57, 1997.

\end{thebibliography}

\end{document}